\documentclass{amsart}
\usepackage{amsmath, amsthm, amssymb,xspace}
\usepackage[breaklinks=true]{hyperref}
\usepackage[mathcal]{euscript}
\usepackage{enumerate}

\newtheorem{theorem}{Theorem}[section]
\newtheorem{lemma}[theorem]{Lemma}
\newtheorem{corollary}[theorem]{Corollary}
\newtheorem{proposition}[theorem]{Proposition}

\theoremstyle{definition}
\newtheorem{remark}[theorem]{Remark}
\newtheorem{definition}[theorem]{Definition}

\newtheorem{example}[theorem]{Example}

\numberwithin{equation}{section}

\newcommand{\C}{\mathbb{C}}
\newcommand{\D}{\mathbb{D}}

\newcommand{\calD}{\mathcal{D}}

\newcommand{\calM}{\mathcal{M}}
\newcommand{\calX}{\mathcal{X}}

\newcommand{\inner}[2]{{\langle#1,#2\rangle}}
\newcommand{\Binner}[2]{{\Big\langle#1,#2\Big\rangle}}
\newcommand{\binner}[2]{{\big\langle#1,#2\big\rangle}}
\newcommand{\Ho}{\mathcal{H}^2_{\omega}}
\renewcommand{\vector}[2]{{\begin{bmatrix}
{#1}\\
\vdots\\
{#2}
\end{bmatrix}}}
\newcommand{\ds}{\displaystyle}

\DeclareMathOperator{\lspan}{{\rm span}}

\DeclareMathOperator{\mult}{{\rm Mult}}

\DeclareMathOperator{\real}{{\rm Re}}

\begin{document}

\title[Inner functions]{Inner functions in weighted Hardy spaces}

\author{Trieu Le}
\address{Department of Mathematics and Statistics, Mail Stop 942, University of Toledo, Toledo, OH 43606}
\email{Trieu.Le2@utoledo.edu}


\begin{abstract}
Inner functions play a central role in function theory and operator theory on the Hardy space over the unit disk. Motivated by recent works of C. B\'en\'eteau et al. and of D. Seco, we discuss inner functions on more general weighted Hardy spaces and investigate a method to construct analogues of finite Blaschke products. 
\end{abstract}

\maketitle

\section{Introduction}
\label{S:intro}
Let $\omega=\{\omega_n\}_{n\geq 0}$ be a sequence of positive real numbers. We denote by $\Ho$ the Hilbert space of all formal power series $f(z)=\sum_{n=0}^{\infty}a_nz^n$ such that
\[\|f\| = \Big(\sum_{n=0}^{\infty}\omega_n\,|a_n|^2\Big)^{1/2}<\infty.\]
The corresponding inner product shall be denoted by $\inner{}{}$. For any $f,g\in\Ho$,
\[\inner{f}{g} = \sum_{n=0}^{\infty}\omega_n\,a_n\bar{b}_n,\] if $f$ is as above and $g(z)=\sum_{n=0}^{\infty}b_nz^n$. Define $e_n(z)=z^n/\sqrt{\omega_n}$ for all integers $n\geq 0$. Then $\{e_n\}_{n\geq 0}$ forms an orthonormal basis for $\Ho$.

Without loss of generality we shall always assume $\omega_0=1$. In addition, we shall restrict our attention to sequences $\omega$ satisfying
\begin{align}
\label{Eqn:weight_condition}
\lim_{n\to\infty}\frac{\omega_{n+1}}{\omega_n}=1.
\end{align}
This condition guarantees that each function in $\Ho$ is holomorphic on the unit disk and the operator of multiplication by $z$ is a bounded operator on $\Ho$. Furthermore, any function holomorphic on a disk of radius strictly larger than $1$ belongs to $\Ho$. It can be shown that $\Ho$ is a reproducing kernel Hilbert space of holomorphic functions on $\D$, that is, for each $\lambda\in\D$, the evaluation map $f\mapsto f(\lambda)$ is a bounded linear functional on $\Ho$. Consequently, there exists a reproducing kernel $K_{\lambda}\in\Ho$ for which
\[f(\lambda) = \inner{f}{K_\lambda}\,\quad\text{ for all } f\in\Ho.\]
We define the function $K(z,\lambda)=K_\lambda(z)=\inner{K_\lambda}{K_z}$ for $z,\lambda\in\D$ and call $K$ the reproducing kernel for $\Ho$. See \cite[Chapter 2]{CowenCRCP1995} for more details on $\Ho$.

A holomorphic function $\varphi$ on $\D$ is called a \textbf{multiplier} for $\Ho$ if for any $f\in\Ho$, the product function $\varphi f$ also belongs to $\Ho$. In that case, the operator $M_{\varphi}$ of multiplication by $\varphi$ is bounded on $\Ho$. We shall denote
\[\|\varphi\|_{\mult(\Ho)} = \|M_{\varphi}\|\] and call it the multiplier norm of $\varphi$ on $\Ho$. Note that since $\|1\|=1$ on $\Ho$, we always have
\[\|\varphi\| \leq \|\varphi\|_{\mult(\Ho)}.\]
It is well known and not difficult to verify that the operator $M_{z}$ of multiplication by $z$ is a weighted forward shift of multiplicity one:
\begin{align}
\label{Eqn:Mz}
M_{z}(e_n) = \sqrt{\frac{\omega_{n+1}}{\omega_{n}}}\,e_{n+1}\quad\text{ for all } n\geq 0.
\end{align}
Condition \eqref{Eqn:weight_condition} shows that $M_{z}$ is bounded on $\Ho$, which implies that any polynomial is a multiplier of $\Ho$. Furthermore, it is not difficulty to verify that $\|M_{z}f\|\geq \|f\|$ for all $f\in\Ho$ if and only if $\omega$ is non-decreasing. In such a case, we say that $\Ho$ has the expansive shift property.

If $\omega_n=1$ for all $n\geq 0$, then $\Ho$ is the usual Hardy space $H^2$ over the unit disk. The norm on $H^2$ is also given by
\[\|f\|_{H^2}^2 = \sup_{0<r<1}\Big(\frac{1}{2\pi}\int_{0}^{2\pi}|f(re^{i\,t})|^2\,dt\Big).\]
As usual, we denote by $H^{\infty}$ the Banach space of all bounded holomorphic functions $f$ on the unit disk with norm
\[\|f\|_{\infty} = \sup_{z\in\D}|f(z)|.\]
It is well known that $H^{\infty}$ is the space of all multipliers for $H^2$. See, for example, the books \cite{DurenAP1970, NikolskiAMS2002, ZhuAMS2007} for details on the Hardy space.

If $\omega_n = \frac{1}{n+1}$ for all $n\geq 0$, then $\Ho$ is the Bergman space $A^2$ of holomorphic functions on the unit disk. For $f\in A^2$, we also have
\[\|f\|^2_{A^2} = \int_{\D}|f(z)|^2\,dA(z),\]
where $dA$ is the normalized Lebesgue measure on $\D$. It is well known that the multiplier space of $A^2$ coincides with $H^{\infty}$. The theory of Bergman spaces have been discussed in different aspects in several books. See, for example, \cite{DurenAMS2004, HedenmalmSpringer2000, ZhuAMS2007}.

In the case $\omega_n=n+1$ for all $n\geq 0$, the space $\Ho$ coincides with the Dirichlet space $\calD$, which consists of holomorphic functions $f$ on the unit disk for which
\[\|f\|^2_{\calD} = \|f\|_{H^2}^2 + \|f'\|^2_{A^2} < \infty.\]
The books \cite{ArcozziNYJM2011, El-FallahCTM2014} provide excellent references on the Dirichlet space.

The Dirichlet space $\calD$ belongs to a larger class of Hilbert space of holomorphic functions, called Dirichlet-type spaces. For any positive measure $\mu$ on $\D$, one define $\calD_{\mu}$ as the space of all holomorphic functions $f$ on $\D$ for which
\[\|f\|^2_{\calD_{\mu}} = \|f\|_{H^2}^2 + \int_{\D}|f'(z)|^2\,d\mu(z) <\infty.\]
If $\mu$ is rotation invariant, then $\calD_{\mu}$ coincides with $\Ho$ for an appropriate weight sequence $\omega$.

Recall that a bounded holomorphic function $f$ on $\D$ is called an inner function if $|f(\zeta)|=1$ for a.e. $|\zeta|=1$. Inner functions play an important role in function theory and operator theory on Hardy spaces. The celebrated Beurling's Theorem asserts that any closed subspace of $H^2$ that is invariant for the operator $M_{z}$ of multiplication by $z$ is given by $\varphi H^2$ for some inner function $\varphi$.  See \cite{ChalendarNWEJM2015} for a recent survey of classical and new results linking inner functions and operator theory. Inner functions are characterized as functions $f\in H^{2}$ such that
\[\|f\|_{\infty} = \|f\|_{H^2}=1.\]
Interestingly, inner functions can also be characterized via the inner product in $H^2$. Indeed, a function $f\in H^2$ is inner if and only if $\|f\|_{H^2}=1$ and $\inner{z^mf}{f}=0$ for all integers $m\geq 1$.

In the case of the Dirichlet space $\calD$, Richter \cite{RichterJRAM1988} showed that any invariant subspace for $M_{z}$ is also generated by a single function that satisfies the same orthogonality properties above.

Invariant subspaces on the Bergman space $A^2$ turn out to be more complicated. Aleman, Richter and Sundberg \cite{AlemanActaMath1996} proved an analogue of Beurling's Theorem for $A^2$: any invariant subspace $\calM$ of $A^2$ is generated by the so-called wandering subspace $\calM\ominus z\calM$. Any unit norm function in this subspace satisfies $\|f\|_{A^2}=1$ and $z^mf\perp f$ for all $m\geq 1$ and is called an $A^2$-inner function. See \cite[Chapters~5 and 9]{DurenAMS2004} and \cite[Chapter~3]{HedenmalmSpringer2000} for a detailed discussion of inner functions on Bergman spaces $A^p$. 


Generalizing the notation of inner functions, several researchers have defined and studied inner functions on weighted Hardy spaces. 

\begin{definition}
A function $f\in\Ho$ is an $\Ho$-\textbf{inner function} if $\|f\|_{\Ho}=1$ and for all integers $m\geq 1$,
\[\inner{z^m f}{f} = 0.\]
Equivalently, $\inner{pf}{f}=p(0)$ for all polynomials $p$.
\end{definition}

Operator-valued inner functions on vector-valued weighted Hardy spaces have also been defined and studied \cite{BallIEOT2013, BallASM2013, OlofssonAiA2007}. In particular, Ball and Bolotnikov \cite{BallASM2013} obtained a realization of inner functions on vector-valued weighted Hardy spaces. In \cite{BallCM2016}, they investigated the expansive multiplier property of inner functions. They obtained a sufficient condition on the weight sequences for which any inner function has the expansive multiplier property. Recently, B\'en\'eteau et al. \cite{BeneteauJLMS2016, BeneteauCMB2017} studied inner functions and examined the connections between them and optimal polynomial approximants. They also described a method to construct inner functions that are analogues of finite Blaschke products with simple zeroes. In \cite{SecoCAOT2019}, Seco discussed inner functions on Dirichlet-type spaces and characterized such functions as those whose norm and multiplier norm are both equal to one. Cheng, Mashreghi and Ross \cite{ChengArXiv2018} introduced and studied the notion of inner functions with respect to a bounded linear operator. Using their language, our $\Ho$-inner functions are their inner functions with respect to the operator of multiplication by $z$. In a recent paper B\'en\'eteau at. al. \cite{BeneteauAMP2019} investigated inner functions on general simply connected domains in the complex plane.

In this note we first obtain several characterizations of inner functions on weighted Hardy spaces $\Ho$, which extend Seco's results. We then describe a construction of $\Ho$-analogues of finite Blaschke products. Our result provides a complete picture of inner functions that are finite linear combinations of kernel functions. In the last section, we discuss a construction of inner functions via reproducing kernels of certain weighted $\Ho$-spaces.

\section{Some characterizations of inner functions}
\label{S:characterization}

Inner functions are closely related to solutions of a certain extremal problem that has been considered by several authors \cite{AlemanActaMath1996,BeneteauCMB2017,HedenmalmJRAM1991}.
Recall that a closed subspace $\calM$ of $\Ho$ is $z$-invariant if $z\calM\subset\calM$. For a function $f\in\Ho$, we will write $[f]$ (or $[f]_{\Ho}$ if there may be a confusion) for the $z$-invariant subspace generated by $f$, which is the closure of all polynomial multiples of $f$ in the norm of $\Ho$.

Let $\{0\}\neq\calM\subsetneq\Ho$ be $z$-invariant. Let $d\geq 0$ be the smallest integer such that $z^d\notin\calM^{\perp}$. Consider the extremal problem
\begin{align}
\label{Eqn:optimal_real_part}
\sup\Big\{\real(g^{(d)}(0)): g\in\calM,\, \|g\|\leq 1\Big.\}
\end{align}
Since $g^{(d)}(0)$ is a positive constant multiple of $\inner{g}{z^d}$, the above extremal problem is equivalent to
\begin{align}
\label{Eqn:optimal_inner}
\sup\Big\{\real(\inner{g}{z^d}): g\in\calM,\, \|g\|\leq 1\Big\}.
\end{align}
It follows from general Hilbert space theory that the problem \eqref{Eqn:optimal_inner} has a unique solution, which is exactly $P_{\calM}(z^d)/\|P_{\calM}(z^d)\|$, where $P_{\calM}(z^d)$ denotes the orthogonal projection of $z^d$ onto $\calM$. The following theorem shows that such projection produces an $\Ho$-inner function and that each $\Ho$-inner function arises in such a fashion. This result is known in many spaces and has appeared in the literature in different forms. The case $d=0$ was detailed in \cite[Theorem 2.2]{BeneteauCMB2017}. The general case is in fact similar and we provide a sketch of the proof.

\begin{theorem}
\label{T:inner_projection_complement}
Let $f$ be a function in $\Ho$ with $\|f\|=1$. Then the following statements are equivalent.
\begin{enumerate}[(a)]
   \item $f$ is an $\Ho$-inner function.
   \item There exist a $z$-invariant subspace $\calM\neq\{0\}$ and an integer $d\geq 0$ such that $\{z^k: 0\leq k\leq d-1\}\subset\calM^{\perp}$, $z^{d}\notin \calM^{\perp}$ and $f$ is a constant multiple of $P_{\calM}[z^d]$.
\end{enumerate}
\end{theorem}

\begin{proof}
Suppose first $f$ is $\Ho$-inner. Let $d\geq 0$ be the largest integer for which $z^d$ divides $f$. It can then be verified that $\inner{z^d}{f}f=P_{[f]}(z^d)$, which shows that (b) holds with $\calM=[f]$.

Conversely, suppose that (b) holds. 
Put $u=P_{\calM}(z^d)$. Then $u\in\calM$ and $z^d-u\in\calM^{\perp}$. For any integer $m\geq 1$, since $\calM$ is $z$-invariant, the function $z^{m}u$ belongs to $\calM$. It follows that
\begin{align}
\label{Eqn:orthogonalityA}
0  = \inner{z^mu}{z^d-u} = \inner{z^{m}u}{z^d} - \inner{z^mu}{u}.
\end{align}
On the other hand, by assumption, $z^d$ divides any function in $\calM$ so we may write $u=z^dv$ for some $v$ holomorphic on the unit disk. Consequently,
\begin{align}
\label{Eqn:orthogonalityB}
\inner{z^mu}{z^d} = \inner{z^{m+d}v}{z^d} = 0
\end{align}
since $m\geq 1$. It follows from \eqref{Eqn:orthogonalityA} and \eqref{Eqn:orthogonalityB} that $\inner{z^mu}{u}=0$, hence $\inner{z^mf}{f}=0$ for all integers $m\geq 1$. By the fact that $f$ is a unit vector, (a) follows.
\end{proof}

Inner functions on the Bergman space have been very well studied. See \cite[Chapter~5]{DurenAMS2004} and \cite[Chapter~3]{HedenmalmSpringer2000} for an excellent treatment of inner functions even in the general setting of $A^p$ spaces. It is well known (see \cite[Section~5.2]{DurenAMS2004} and \cite[Section~3.4]{HedenmalmSpringer2000}) that each $A^2$-inner function has an expansive multiplier property. That is, if $f$ is an $A^2$-inner function, then
\[\|p\|_{A^2}\leq \|pf\|_{A^2}\] for all polynomials $p$. Ball and Bolotnikov\cite{BallCM2016} exhibited sufficient conditions on the weight sequence $\omega$ for which any inner function in $\Ho$ has the expansive multiplier property.

In \cite{RichterJOT1992}, among other things, Richter and Sundberg found a close connection between inner functions on the Dirichlet space and its multipliers: if $f$ is $\calD$-inner, then
\[\|f\|_{\mult(\calD)}=1.\]
In particular, each $\calD$-inner function is a contractive multiplier. This is quite different from $A^2$-inner functions.

Seco \cite{SecoCAOT2019} recently obtained the following interesting characterization of inner functions on Dirichlet-type spaces generated by rotation-invariant measures $\mu$.
\begin{theorem}[\cite{SecoCAOT2019}]
\label{T:Seco}
Let $f\in\calD_{\mu}=\Ho$. Then $f$ is $\calD_{\mu}$-inner whenever
\begin{enumerate}[(a)]
  \item $\|f\|_{\calD_{\mu}}=1$ and for all integers $k\geq 1$ and all $\lambda\in\C$, \[\|fg_{k,\lambda}\|^2_{\calD_{\mu}} \leq\omega_k + |\lambda|^2,\] where $g_{k,\lambda}(z)=z^k+\lambda$, and this holds true whenever
   \item $\|f\|_{\calD_{\mu}}=\|f\|_{\mult(\calD_{\mu})}=1$.   
\end{enumerate}
\end{theorem}

Note that condition (a) in Theorem \ref{T:Seco} is considerably weaker than the condition $\|f\|_{\mult(\calD_{\mu})}=1$ in (b). As Seco pointed out, his proof of (b)$\Rightarrow$(a)$\Rightarrow \calD_{\mu}$-inner works for any weighted Hardy space with the expansive shift property. On the other hand, for the Dirichlet space, using Richter-Sundberg's result, we see that (a) actually implies (b).

On any $\Ho$, the function $e_k=z^k/\sqrt{\omega_k}$ is always $\Ho$-inner. The multiplication operator $M_{e_k}$ is a weighted shift:
\[M_{e_k}(e_n) = \sqrt{\frac{\omega_{n+k}}{\omega_{n}\,\omega_k}}\,e_{n+k}\quad\text{for all } n\geq 0.\]
Consequently, $e_k$ is a contractive multiplier if and only if
\begin{align}
\label{Eqn:contractive_monomial_inner}
\omega_{n+k}\leq \omega_{n}\,\omega_k\quad\text{ for all integers } n\geq 0.
\end{align}
On the other hand, $e_k$ is an expansive multiplier if and only if
\begin{align}
\label{Eqn:expansive_monomial_inner}
\omega_{n+k}\geq \omega_n\,\omega_k\quad\text{ for all integers } n\geq 0.
\end{align}

\begin{example}
Consider the weight sequence $\omega$ defined by $\omega_1=\sqrt{2}$ and $\omega_n=n+1$ for all $n\neq 1$. This weight sequence has the expansive shift property and it is almost identical to the weight sequence generating the Dirichlet space $\calD$ except the value of $\omega_1$. Since $\omega_2 > \omega_1^2$, the $\Ho$-inner function $e_1(z)=z/\sqrt{\omega_1}$ has multiplier norm bigger than $1$. Indeed,
\begin{align*}
\|z\cdot e_1\|^2_{\Ho}=\Big\|\frac{z^2}{\sqrt{\omega_1}}\Big\|^2_{\Ho} = \frac{\|z^2\|^2_{\Ho}}{\omega_1} = \frac{\omega_2}{\omega_1}>\omega_1 = \|z\|^2_{\Ho}.
\end{align*}
\end{example}

\begin{example}
Now consider the weight sequence $\omega_1=1/\sqrt{2}$ and $\omega_n=1/(n+1)$ for all $n\neq 1$, which consists of the reciprocals of the sequence in the previous example. This weight sequence is almost identical to the weight sequence generating the Bergman space $A^2$ except the value of $\omega_1$. Since $\omega_2 < \omega_1^2$, the $\Ho$-inner function $e_1(z)=z/\sqrt{\omega_1}$ does \textit{not} have the expansive multiplier property.
\end{example}

The above examples show that the expansive and contractive properties are quite sensitive under very small changes in the norms. Interestingly, we show in the following result that on any $\Ho$-space, a function is inner if and only if it satisfies a certain modified expansive and contractive multiplier property. We indeed obtain several equivalent characterizations of $\Ho$-inner functions. 

For any integer $k\geq 1$ and any complex number $\zeta\in\C$, denote $g_{k,\lambda}(z)=z^k+\lambda$ for $z\in\C$.

\begin{theorem}
\label{T:inner_characterization}
Let $f$ be an element of $\Ho$. Then the following statements are equivalent.
\begin{enumerate}[(i)]
   \item $f$ is $\Ho$-inner.
   \item $\|f\|=1$ and for any integer $k\geq 1$, one of the following conditions holds:
       \begin{enumerate}[(C1)]
         \item There exists a constant $C_k$ such that
   \[\|fg_{k,\lambda}\|^2 \leq C_k+|\lambda|^2\quad\text{for all }\lambda\in\C.\]
         \item There exists a constant $D_k$ such that
   \[\|fg_{k,\lambda}\|^2 \geq D_k+|\lambda|^2\quad\text{for all }\lambda\in\C.\]
       \end{enumerate}
   \item $\|f\|=1$ and for any polynomial $p\in\C[z]$, we have
   \[|p(0)| \leq \|pf\|.\]
\end{enumerate} 
\end{theorem}

\begin{remark}
Since a general $\Ho$-inner function may fail to be an expansive multiplier, we may perhaps consider \textit{(iii)} as an analogue of the expansive multiplier property. While the proof is not difficult, it is quite curious to us that the implication \textit{(iii)}$\Rightarrow$\textit{(i)}, even for the Hardy space, does not seem to appear in the references that we have known.
\end{remark}

\begin{proof}
If $\|f\|=1$, for $k\geq 1$, we compute
\begin{align*}
\varphi_k(\lambda)
& =\|fg_{k,\lambda}\|^2 - |\lambda|^2\\
& = \|z^kf\|^2 + |\lambda|^2\|f\|^2 + 2\Re(\bar{\lambda}\inner{z^kf}{f}) - |\lambda|^2 \\
& = \|z^kf\|^2 + 2\Re(\bar{\lambda}\inner{z^kf}{f}).
\end{align*}
For any real number $t$, setting $\lambda=t\inner{z^kf}{f}$ gives
\[\varphi_k\big(t\inner{z^kf}{f}\big) = \|z^kf\|^2 + 2t\,|\inner{z^kf}{f}|^2.\]
It follows that $\varphi_k$ is bounded above or bounded bellow if and only if $\inner{z^kf}{f}=0$. This shows that $(i)$ and $(ii)$ are equivalent.

We now prove the implication $(i)\Rightarrow(iii)$. Assume that $f$ is $\Ho$-inner. Let $p(z)=a_0+a_1z+\cdots+a_nz^n$ be any polynomial. We then have
\begin{align*}
\inner{pf}{f} & = \sum_{j=0}^{n}a_j\inner{z^jf}{f} = a_0 = p(0).
\end{align*}
Cauchy-Schwarz's inequality implies
\begin{align*}
|p(0)| = |\inner{pf}{f}| \leq\|pf\|\cdot\|f\| = \|pf\|.
\end{align*}
Therefore, $(i)$ implies $(iii)$. On the other hand, if $(iii)$ holds, then for each integer $k\geq 1$, condition (C2) is satisfied with $D_k=0$. Consequently, $(i)$ holds.
\end{proof}

\begin{corollary} Let $f$ be an element of $\Ho$. Then $f$ is $\Ho$-inner if and only if for any $k\geq 1$, there exist constants $C_k$ and $D_k$ such that 
\begin{align}
\label{Eqn:inner_boundedness}
D_k+|\lambda|^2 \leq \|fg_{k,\lambda}\|^2 \leq C_k+|\lambda|^2 \quad\text{ for all } \lambda\in\C.
\end{align}
\end{corollary}

\begin{proof}
By Theorem \ref{T:inner_characterization}, we only need to show that if \eqref{Eqn:inner_boundedness} holds, then $\|f\|=1$. Note that
\[\lim_{|\lambda|\to\infty}\frac{\|fg_{k,\lambda}\|^2}{|\lambda|^2} = \lim_{|\lambda|\to\infty}\frac{\|z^kf\|^2 + 2\Re(\bar{\lambda}\inner{z^kf}{f})}{|\lambda|^2} + \|f\|^2 = \|f\|^2.\]
It follows that if \eqref{Eqn:inner_boundedness} holds for some integer $k\geq 1$, then $\|f\|=1$.
\end{proof}

In the case $f$ is a multiplier of $\Ho$, the following result provides a necessary and sufficient condition for $f$ to be an inner function via the action of the operator $M_f^{*}M_f$ on constant functions.

\begin{theorem}
\label{T:inner_operator_char}
Let $f$ be a multiplier of $\Ho$. Then the following statements are equivalent.
\begin{itemize}
   \item[(a)] $f$ is $\Ho$-inner.
   \item[(b)] $M_{f}^{*}M_{f}1=1$.
   \item[(c)] There exists a positive operator $A$ on $z\Ho$ such that
   \[M_{f}^{*}M_{f} = I_{\C\mathbf{1}}\oplus A,\] where $I_{\C\mathbf{1}}$ denotes the identity operator on the one-dimensional space of constant functions.
\end{itemize}
\end{theorem}

\begin{proof}
Recall that $\{e_{m}(z)=z^m/\sqrt{\omega_m}: m=0,1,\ldots\}$ is an orthonormal basis for $\Ho$. Let us compute
\begin{align*}
M_{f}^{*}M_{f}1 & = \sum_{m=0}^{\infty}\inner{M_{f}^{*}M_{f}1}{e_m}e_m = \sum_{m=0}^{\infty}\inner{M_f1}{M_fz^m}\frac{z^m}{\omega_m}\\
& = \sum_{m=0}^{\infty}\inner{f}{z^mf}\frac{z^m}{\omega_m}.
\end{align*}
It follows that $M_{f}^{*}M_{f}1=1$ if and only if $\|f\|=1$ and $\inner{f}{z^mf}=0$ for all $m\ge 1$. Therefore,  (a) and (b) are equivalent. On the other hand, since $\Ho=\C\mathbf{1}\oplus z\Ho$, we see that (b) and (c) are equivalent.   
\end{proof}

\begin{remark}
When $\Ho$ is the Hardy space or one of the weighted Bergman spaces over the disk, the operator $M_{f}^{*}M_{f}$ can be identified as the Toeplitz operator $T_{|f|^2}$. Theorem \ref{T:inner_operator_char} asserts that in such spaces, a function $f$ is inner if and only if constant functions belong to $\ker(I-T_{|f|^2})$. On the Hardy space, this is equivalent to the fact that $|f|=1$ a.e. on the unit circle. This, of course, does not hold true for inner functions on the Bergman nor any general $\Ho$ space.
\end{remark}

\begin{remark} Statement (b) in Theorem \ref{T:inner_operator_char} essentially says that $M_{f}^{*}M_{f}$ is the identity operator on the subspace of constant functions. We mention in passing that an analogous characterization of operator-valued inner functions can also be obtained with a similar proof. 
\end{remark}

\section{Construction of finite $\Ho$-Blaschke products}
\label{S:construction}

In this section, we shall develop a method to construct analogues of finite Blaschke products. These functions turn out to be multipliers of $\Ho$. We provide here a short discussion which shows that under the assumption \eqref{Eqn:weight_condition} on the weight sequences, many \textit{nice} functions serve as multipliers of $\Ho$. We have seen in the introduction that any polynomial is a multiplier of $\Ho$. More is true as we shall see below.

Recall from \eqref{Eqn:Mz} that
\[M_{z}(e_n) = \sqrt{\frac{\omega_{n+1}}{\omega_n}}\,e_{n+1}\ \text{ for all } n\geq 0.\]
It follows that for all $m\geq 1$,
\[M_{z^m}(e_n) = \sqrt{\frac{\omega_{n+m}}{\omega_n}}\,e_{n+m}\ \text{ for all } n\geq 0\]
and hence,
\[\|M_{z^m}\| = \sup\Big\{\sqrt{\frac{\omega_{n+m}}{\omega_n}}: n\geq 0\Big\}.\]
Condition \eqref{Eqn:weight_condition} then implies
\begin{align}
\label{Eqn:limit_norm_Mzm}
\lim_{m\to\infty}\|M_{z^m}\|^{1/m} = \lim_{m\to\infty}\Big(\sup\Big\{\sqrt{\frac{\omega_{n+m}}{\omega_n}}: n\geq 0\Big\}\Big)^{1/m} = 1.
\end{align}

\begin{lemma}
\label{L:multiplier}
Let $\varphi$ be a function that is holomorphic on a disk centered at the origin with radius strictly larger than one. Then $\varphi$ is a multiplier of $\Ho$. 
\end{lemma}

\begin{proof}
Write
\[\varphi(z) = \sum_{m=0}^{\infty}b_mz^m.\] 
The hypothesis implies that $\ds\rho=\limsup_{m\to\infty}|b_m|^{1/m}<1$. Using this together with \eqref{Eqn:limit_norm_Mzm}, we conclude
\[\limsup_{m\to\infty}(|b_m|\,\|M_{z^m}\|)^{1/m} = \rho < 1.\]
Consequently,
\[\sum_{m=0}^{\infty}|b_m|\,\|M_{z^m}\|<\infty,\]
which implies that the series $\sum_{m=0}^{\infty}b_m M_{z^m}$ converges to a bounded operator on $\Ho$.  This shows that $M_{\varphi}$ is bounded as desired.
\end{proof}

Using Shapiro-Shields's determinants \cite{ShapiroMZ1962}, Beneteau et al. \cite{BeneteauCMB2017} described a method of constructing $\Ho$-inner functions with a prescribed finite set of simple zeros in the disk. They call such inner functions analogues of finite Blaschke products. In this section, we investigate this problem with a different approach. We consider inner functions with a prescribed finite set of zeros (with finite multiplicities) and we shall explain why these functions should be considered natural analogues of finite Blaschke products. Our approach makes use of reproducing kernel functions.

Recall that for each $\lambda\in\D$, we use $K_{\lambda}$  to denote the reproducing kernel function of $\Ho$ at $\lambda$. For $\Ho=H^2$, we have
\[K^{H^2}_{\lambda}(z) = \frac{1}{1-\overline{\lambda}z}\quad\text{for } z\in\D\] and for $\Ho=A^2$, we have
\[K^{A^2}_{\lambda}(z) = \frac{1}{(1-\overline{\lambda}z)^2}\quad\text{for } z\in\D.\]

Any finite Blaschke product with simple non-zero roots is a unimodular constant multiple of
\[b(z) = z^{d}\cdot\frac{z-z_1}{1-\bar{z}_1z}\cdots\frac{z-z_s}{1-\bar{z}_sz},\]
where $d\geq 0$ and $z_1,\ldots, z_s$ are distinct points in $\D\backslash\{0\}$. Such $b$ is $H^2$-inner. How would one construct an analogue of $b$ in more general $\Ho$? Our approach is to use the partial fractions decomposition of $b$:
\begin{align}
\label{Eqn:partial_decomp_Blacshke}
b(z) & = p(z) + \frac{c_1}{1-\bar{z}_1z} + \cdots+\frac{c_s}{1-\bar{z}_sz}\notag\\
& = p(z) + c_1K^{H^2}_{z_1}(z) + \cdots + c_sK^{H^2}_{z_s}(z).
\end{align}
where $p$ is a polynomial of degree $d$ and $c_1,\ldots, c_s$ are non-zero complex values. Motivated by such a decomposition, we seek $\Ho$-inner functions that can be written as a finite combination of kernel functions. It turns out that such inner functions can be computed using Shapiro-Shield's determinants, which recovers Beneteau et al.'s results. To illustrate our ideas, let us consider several examples.

\begin{example}
\label{E:monimial_Blaschke}
It follows directly from the definition that for each integer $d\geq 1$, the function $z^d/\sqrt{\omega_d}$ is $\Ho$-inner. Such a function is an $\Ho$-analogue of the $H^2$-inner function $z^d$.
\end{example}

\begin{example}
\label{E:single_Blaschke_factor} 
Let $a\in\D\backslash\{0\}$. We seek $\Ho$-analogue of the Blaschke factor \[b_1(z)=\frac{z-a}{1-\bar{a}z}.\]
To do that, we need to determine two non-zero constants $\alpha,\beta\in\C$ such that
$B_1 = \alpha + \beta K_{a}$ is an $\Ho$-inner function. The conditions for $B_1$ to be inner are $\|B_1\|=1$ and 
\begin{align*}
\inner{\alpha z^j + \beta z^j K_{a}}{\alpha + \beta K_{a}} = 0 \quad\text{ for all } j\geq 1. 
\end{align*}
Since $\inner{z^j}{1}=0$ and $K_{a}$ is the reproducing kernel at $a$, we then obtain
\[\alpha\overline{\beta} a^j + \beta\overline{\beta} a^{j} K_{a}(a) = 0,\]
which implies $\alpha = -\beta K_{a}(a)$. Consequently, $B_1$ is a constant multiple of the function $K_a(a) - K_a$. That is, any $\Ho$-analogue of $b_1$ is of the form
\[B_1=\mu\frac{K_a(a)-K_a}{\|K_a(a)-K_a\|} = \mu\frac{K_a(a)-K_a}{\sqrt{K_a(a)(K_a(a)-1)}},\] where $|\mu|=1$. Of course, this is very well known. See, for example \cite[p.~125]{DurenAMS2004}.

Note that $B_1(a)=0$. Does $B_1$ admit any extraneous zeros in the closed unit disk? By the formula for the reproducing kernel, we have
\begin{align*}
K_a(a)-K_a(z) & = \sum_{n=1}^{\infty}\frac{\bar{a}^n(a^n-z^n)}{\omega_n} = \bar{a}(a-z)\Big(\frac{1}{\omega_1} + \sum_{n\geq 2}\frac{\bar{a}^{n-1}(a^n-z^n)}{(a-z)\omega_n}\Big)
\end{align*}
Consequently,
\begin{align*}
\Big|K_a(a)-K_a(z)\Big| \geq |a|\cdot|a-z|\cdot\Big(\frac{1}{\omega_1} - |a|\sum_{j=0}^{\infty}\frac{(j+2)|a|^{j}}{\omega_{j+2}}\Big).
\end{align*}
It follows that there exists $\delta>0$ such that for all $|a|<\delta$ and $|z|\leq 1$,
\[\Big|K_a(a)-K_a(z)\Big| \geq \frac{|a|}{2\omega_1}|a-z|.\]
This implies that $a$ is the unique zero of $B_1$ in the closed unit disk provided that $|a|<\delta$.

When $\Ho$ is the Hardy space or the Bergman space, it follows from the explicit formula of the reproducing kernel functions that $B_1$ does not have any other zero except $a$ and this is true for any $a\in\D$. Surprisingly, for certain weighted Bergman spaces, $B_1$ does possesses extraneous zeros different from $a$. In general, Hedenmalm and Zhu \cite{HedenmalmCVTA1992} and Weir \cite{WeirPAMS2002, WeirPJM2003} showed that  many properties of (un-weighted) Bergman space inner functions may not hold in the setting of weighted Bergman spaces. 

\end{example}

\begin{example}
\label{E:two_root_Blaschke}
Let $a\in\D\backslash\{0\}$. We now seek $\Ho$-analogue of the Blaschke factor \[b_2(z) = z\cdot\frac{z-a}{1-\bar{a}z}.\]
We need to find three constants $\alpha, \beta,\gamma\in\C$, where $\alpha$ and $\gamma$ are non-zero such that $g=\alpha z + \beta + \gamma K_a$ is $\Ho$-inner. The conditions for $g$ to be inner are given as
\begin{align*}
1 & = \|g\| = \|\alpha z + \beta + \gamma K_a\|,\\
0 & = \inner{zg}{g} =  \bar{\alpha}\inner{zg}{z} + \bar{\gamma}\inner{zg}{K_a} = \bar{\alpha}(\beta + \gamma)\omega_1 + \bar{\gamma}ag(a)\\
0 & = \inner{z^jg}{g} = \inner{z^jg}{\gamma K_a} = \bar{\gamma}a^{j}g(a)\quad\text{for all } j\geq 2.
\end{align*}
The last condition forces $g(a)=0$, which together with the second condition gives $\bar{\alpha}(\beta+\gamma)=0$. Since $\alpha\neq 0$, we conclude that $\beta+\gamma=0$, which is equivalent to $g(0)=0$. It then follows that $g$ is a unimodular multiple of the function
\[B_2(z)=\dfrac{\big(1-K_a(a)\big)z - a(1-K_a)}{\|\big(1-K_a(a)\big)z - a(1-K_a)\|_{\Ho}} = \dfrac{\big(1-K_a(a)\big)z - a(1-K_a)}{\sqrt{(K_a(a)-1)((K_a(a)-1)\omega_1 - |a|^2)}}.\]
\end{example}

\begin{example}
\label{E:repeated_root_Blaschke}
Let $a\in\D\backslash\{0\}$. We look for $\Ho$-analogue of $\big(\frac{z-a}{1-\bar{a}z}\big)^2$. Write
\begin{align*}
\Big(\frac{z-a}{1-\bar{a}z}\Big)^2 & = \frac{1}{\bar{a}^2} + \Big(a^2-\frac{1}{\bar{a}^2}\Big)\cdot\frac{1}{1-\bar{a}z} + \frac{(1-|a|^2)^2}{\bar{a}}\cdot\frac{z}{(1-\bar{a}z)^2}\\
& = \frac{1}{\bar{a}^2} + \Big(a^2-\frac{1}{\bar{a}^2}\Big)K_a^{H^2} + \frac{(1-|a|^2)^2}{\bar{a}}\frac{\partial K_a^{H^2}}{\partial\bar{a}}.
\end{align*}
Consequently, an $\Ho$-analogue should admit the form
\[B(z) = \alpha + \beta K_a + \gamma\frac{\partial K_a}{\partial\bar{a}}.\]
While it is possible to use the calculations similar to the previous examples to determine $\alpha, \beta$ and $\gamma$, we shall discuss this example later after we have proved our result.
\end{example}

\subsection{Gram determinant}
\label{SS:Gram_determinant}
One of our tools to construct analogues of higher order Blaschke products is the notion of the Gram determinant of a finite set of vectors, which we now recall. For a collection of vectors $v_1,\ldots, v_s$ in an inner-product space, let $G(v_1,\ldots,v_s)$ denote the Gram matrix
\[G(v_1,\ldots,v_s) = 
\begin{bmatrix}
\inner{v_1}{v_1} & \cdots & \inner{v_1}{v_s}\\
\vdots & \cdots & \vdots\\
\inner{v_s}{v_1} & \cdots & \inner{v_s}{v_s}
\end{bmatrix}.\]
The Gram determinant is then $\det(G(v_1,\ldots,v_s))$. Since the matrix is positive semi-definite, $\det(G(v_1,\ldots,v_s))\geq 0$. Furthermore, the vectors $\{v_1,\ldots, v_s\}$ are linearly independent if and only if $\det(G(v_1,\ldots,v_s))>0$.

For any vector $v$, we denote by $D(v,v_1,\ldots,v_s)$ the vector
\begin{align}
\label{Eqn:defn_projection}
D(v;\, v_1,\ldots,v_s) = \det\begin{pmatrix}
v & \inner{v}{v_1} & \cdots & \inner{v}{v_s}\\
v_1 & \inner{v_1}{v_1} & \cdots & \inner{v_1}{v_s}\\
\vdots & \vdots & \cdots & \vdots\\
v_s & \inner{v_s}{v_1} & \cdots & \inner{v_s}{v_s}
\end{pmatrix}.
\end{align}
Here the determinant is computed in the usual way even though the first column consists of vectors. Note that $D(v; v_1,\ldots,v_s)\perp v_j$ for all $1\leq j\leq s$ and
\begin{align*}
\|D(v;\,v_1,\ldots,v_s)\|^2 & = \det(G(v_1,\ldots,v_s))\binner{D(v;\,v_1,\ldots,v_s)}{v}.
\end{align*}
If the vectors $\{v_1,\ldots,v_s\}$ are linearly independent so that $\det(G(v_1,\ldots,v_s))>0$, then it is well known that
\[u = \frac{1}{\det(G(v_1,\ldots,v_s))}D(v;\, v_1,\ldots, v_s)\] is the orthogonal projection of $v$ onto the subspace $\{v_1,\ldots,v_s\}^{\perp}$. To see this, note that $u\perp v_j$ for all $1\leq j\leq s$ and $v-u$ is a linear combination of $\{v_1,\ldots, v_s\}$.

\subsection{$\Ho$-Blaschke products}

Recall that the reproducing kernel $K_\lambda(z)$, as a function of two variables, is holomorphic in $z$ and conjugate holomorphic in $\lambda$. For any integer $\ell\geq 0$, define
\[K^{(\ell)}_{\lambda}(z) = \frac{\partial^{\ell}}{\partial\bar{\lambda}^{\ell}}(K(z,\lambda)).\]
Then $K^{(\ell)}_{\lambda}$ is the reproducing kernel for the $\ell$th derivative of functions in $\Ho$ at $\lambda$. That is, for all $h\in\Ho$ and $\lambda\in\D$,
\[h^{(\ell)}(\lambda) = \inner{h}{K^{(\ell)}_{\lambda}}.\]

We denote by $\calX_{\lambda}$ the linear subspace spanned by  the set of linearly independent functions $\{K_\lambda,K^{(1)}_\lambda, K^{(2)}_\lambda,\ldots\}$. For any $f\in\calX_{\lambda}\backslash\{0\}$, there exist a unique integer $m\geq 0$ and constants $c_0,\ldots, c_{m}$ with $c_m\neq 0$ such that $f = \sum_{j=0}^{m}c_jK_\lambda^{(j)}$. We shall call $m$ the degree of $f$. It is clear that $\calX_{0}$ is the space of all polynomials and in that case we recover the usual notion of the degree of a polynomial.

\begin{proposition}
\label{P:linear_independency}
Let $z_1,\ldots, z_s$ be distinct points in $\D$. Then the subspaces $\{\calX_{z_j}:1\leq j\leq s\}$ are mutually linearly independent in the sense that if 
$h_j\in\calX_{z_j}$ and $h_1+\cdots + h_s=0$, then $h_j=0$ for all $j$.
\end{proposition}

\begin{proof}
For each $j$, there exists a polynomial $p_j$ for which
\[\inner{f}{h_j} = \big(p_j(D)f\big)(z_j)\quad\text{ for all  } f\in\Ho,\]
where $D=\frac{d}{dz}$. If $h_t\neq 0$ for some $1\leq t\leq s$, then $p_t$ is not identically zero. There then exists a holomorphic polynomial $f$ such that $(p_j(D)f)(z_j)=0$ for all $j\neq t$ and $(p_t(D)f)(z_t)\neq 0$. Indeed, one may take
\[f(z)=(z-z_t)^{d_t}\prod_{j\neq t}(z-z_j)^{d_j+1},\]
where $d_j = 0$ if $p_j$ is the zero polynomial and $d_j=\deg(p_j)$ if $p_j\neq 0$. We then have 
\[\inner{f}{h_1+\cdots+h_s} = \big(p_1(D)f\big)(z_1)+\cdots+\big(p_s(D)f\big)(z_s) = (p_t(D)f)(z_t) \neq 0,\]
which implies $h_1+\cdots+h_s\neq 0$. The conclusion of the proposition then follows.
\end{proof}

Let $f$ be a multiplier of $\Ho$. For $\lambda\in\D$, it is well known that $M_{f}^{*}K_{\lambda}=\overline{f(\lambda)}K_{\lambda}$. More generally, let us find a formula for $M_{f}^{*}(K^{(\ell)}_\lambda)$ for any integer $\ell\geq 0$. For $h\in\Ho$, we have
\begin{align*}
\inner{h}{M_{f}^{*}(K^{(\ell)}_{\lambda})} & = \inner{h\,f}{K^{(\ell)}_{{\lambda}}} = (h\,f)^{(\ell)}({\lambda})\\
& = \sum_{j=0}^{\ell}\binom{\ell}{j}\,h^{(\ell-j)}({\lambda})\,f^{(j)}({\lambda})\\
& = \sum_{j=0}^{\ell}\binom{\ell}{j}\,f^{(j)}({\lambda})\inner{h}{K^{(\ell-j)}_{\lambda}}\\
& = \Big\langle h, \sum_{j=0}^{\ell}\binom{\ell}{j}\overline{f^{(j)}({\lambda})}\,K^{(\ell-j)}_{\lambda}\Big\rangle.
\end{align*}
Consequently,
\begin{align}
\label{Eqn:adjointM_f_Kernel}
M_{f}^{*}(K^{(\ell)}_{\lambda}) = \sum_{j=0}^{\ell}\binom{\ell}{j}\overline{f^{(j)}({\lambda})}\,K^{(\ell-j)}_{\lambda}.
\end{align}
It follows that the subspace $\calX_{{\lambda}}$ is invariant under $M_f^{*}$ and with respect to the ordered basis $\{K_{\lambda},\ldots,K^{(m)}_{\lambda},\ldots\}$, the operator $M_{f}^{*}$ has an upper triangular matrix representation:
\begin{align*}
\begin{pmatrix}
\overline{f({\lambda})} & \overline{f'({\lambda})} & \overline{f''({\lambda})} & \cdots & \overline{f^{(m)}({\lambda})} & \cdots\\
0 & \overline{f({\lambda})} & 2\,\overline{f'({\lambda})} & \cdots & m\,\overline{f^{(m-1)}({\lambda})} & \cdots\\
0 & 0 & \overline{f({\lambda})} & \cdots & \binom{m}{2}\,\overline{f^{(m-2)}({\lambda})} & \cdots\\
\vdots & \vdots & \vdots & \cdots & \vdots  & \cdots\\
0 & 0 & 0 & \cdots & \overline{f({\lambda})} & \cdots
\end{pmatrix}.
\end{align*}
From this we see that if there exist complex numbers $c_0,\ldots, c_m$ with $c_m\neq 0$ such that
\[M_{f}^{*}\Big(\sum_{j=0}^{m}c_j\,K^{(j)}_{\lambda}\Big) = 0,\]
then $f({\lambda})=\cdots=f^{(m)}({\lambda})=0$.
By a similar argument, we see that if
\[M_{f}^{*}\Big(\sum_{j=0}^{m}c_j\,K^{(j)}_{\lambda}\Big) = K_{{\lambda}},\]
then either $m=0$, or $m\geq 1$ and $f({\lambda})=\cdots=f^{(m-1)}({\lambda})=0$.

Let $Z=\{z_0=0, z_1,\ldots, z_s\}$ be a collection of distinct points on the unit disk. We would like to construct inner functions in $\Ho$ that belong to the sum 
\[\calX_{Z}=\calX_{z_0}+\calX_{z_1}+\cdots+\calX_{z_s}.\]
The sum above is direct by Proposition \ref{P:linear_independency} but non-orthogonal. If $d_0, \ldots, d_s$ are non-negative integers, then it is clear from the determinant expansion that the non-zero vector
\[\nu=D\big(z^{d_0};\,1,z,\ldots,z^{d_0-1}, K_{z_1},\ldots,K_{z_1}^{(d_1)},\ldots, K_{z_s},\ldots,K_{z_s}^{(d_s)}\big)\] belongs to $\calX_{Z}$. On the other hand, by Theorem \ref{T:inner_projection_complement}, $\nu/\|\nu\|_{\Ho}$ is an inner function. It turns out that any inner function in $\calX_{Z}$ is a unimodular constant multiple of such a ratio.

\begin{theorem}
\label{T:directsum_inner}
Let $Z=\{z_0=0, z_1,\ldots, z_s\}$ be a collection of distinct points on the unit disk. Suppose $f_j\in\calX_{z_j}\backslash\{0\}$ and $d_j=\deg(f_j)$ for $0\leq j\leq s$. If $B=f_0+\cdots+f_s$ is $\Ho$-inner, then $B$ is a constant multiple of the vector
\[D\big(z^{d_0};\,1,z,\ldots,z^{d_0-1}, K_{z_1},\ldots,K_{z_1}^{(d_1)},\ldots, K_{z_s},\ldots,K_{z_s}^{(d_s)}\big)\] as defined in \eqref{Eqn:defn_projection}. Furthermore, $B/b$ is holomorphic on a neighborhood of the closed unit disk, where
\[b(z) = z^{d_0}\prod_{j=1}^{s}\Big(\frac{z-z_j}{1-\bar{z}_jz}\Big)^{d_j}.\]
We shall call $B$ an $\Ho$-analogue of the Blaschke product $b$.
\end{theorem}

\begin{proof}
By Lemma \ref{L:multiplier}, $B$ is a multiplier of $\Ho$. Since $B$ is inner, Theorem \ref{T:inner_operator_char} asserts that $M_{B}^{*}M_{B}(1)=1$, which gives
\[M_{B}^{*}(f_0)+\cdots+M_{B}^{*}(f_s) = 1 = K_{z_0}.\]
By Proposition \ref{P:linear_independency}, $M_{B}^{*}(f_j)=0$ for all $1\leq j\leq s$, and $M_{B}^{*}(f_0)=K_{z_0}$. Since $d_j=\deg(f_j)$, the discussion preceding the theorem implies that
\[B(z_j)=\cdots = B^{(d_j)}(z_j)=0\] for all $1\leq j\leq s$, and for $j=0$, either $d_0=0$, or $d_0\geq 1$ and $B(0)=\cdots=B^{(d_0-1)}(0)=0$. Consequently, $B/b$ is holomorphic on a neighborhood of the closed unit disk and $B$ belongs to the orthogonal complement of the subspace
\[\calM=\lspan\Big(\{1,z,\ldots,z^{d_0-1}\}\cup\{K_{z_1},\ldots,K_{z_1}^{(d_1)}\}\cup\cdots\cup\{K_{z_s},\ldots,K_{z_s}^{(d_s)}\}\Big).\]
In the case $d_0=0$, the first set in the above union is understood to be empty. Since $B$ is a linear combination of $z^{d_0}$ and functions in $\calM$, there is a constant $\gamma$ and $h\in\calM$ such that $B=\gamma\,z^{d_0}+h$. We then have
\[B = P_{\calM^{\perp}}(f) = P_{\calM^{\perp}}(\gamma\,z^{d_0}+h) = \gamma\, P_{\calM^{\perp}}(z^{d_0}).\]
By the discussion in subsection \ref{SS:Gram_determinant}, the conclusion of the theorem now follows.
\end{proof}

\begin{example} We now revisit Example \ref{E:repeated_root_Blaschke}. Let $a\in\D\backslash\{0\}$. We would like to construct an $\Ho$-inner function of the form
\[B(z) = \alpha + \beta K_a + \gamma\frac{\partial K_a}{\partial \bar{a}}\] where $\alpha, \beta, \gamma$ are complex constants with $\gamma\neq 0$. Theorem \ref{T:Ho_inner_analogue} asserts that $B$ is a constant multiple of the vector
\begin{align*}
D\big(1;K_a,K_a^{(1)}\big) & = \det\begin{pmatrix}
1 & \inner{1}{K_a} & \inner{1}{K_a^{(1)}}\\
K_a & \inner{K_a}{K_a} & \inner{K_a}{K_a^{(1)}}\\
K_a^{(1)} & \inner{K_a^{(1)}}{K_a} &  \inner{K_a^{(1)}}{K_a^{(1)}}
\end{pmatrix}\\
& = \det\begin{pmatrix}
1 & 1 & 0\\
K_a & K_a(a) & \overline{K_a^{(1)}(a)}\\
K_a^{(1)} & K_a^{(1)}(a) & \|K_a^{(1)}\|^2
\end{pmatrix}.
\end{align*}
If $\Ho=H^2$, then the above formula becomes
\begin{align*}
\det\begin{pmatrix}
1 & 1 & 0\\
\frac{1}{1-\bar{a}z} & \frac{1}{1-|a|^2} & \frac{\bar{a}}{(1-|a|^2)^2}\\
\frac{z}{(1-\bar{a}z)^2} & \frac{a}{(1-|a|^2)^2} & \frac{1+|a|^2}{(1-|a|^2)^3}
\end{pmatrix} 
& 
 =\frac{1}{(1-|a|^2)^4}\Big\{1-\frac{1-|a|^4}{1-\bar{a}z} + \frac{(1-|a|^2)^2\bar{a}z}{(1-\bar{a}z)^2}\Big\}\\
 & = \frac{\bar{a}^2}{(1-|a|^2)^4}\,\Big(\frac{a-z}{1-\bar{a}z}\Big)^2,
\end{align*}
as expected.
\end{example}

\section{Inner functions via reproducing kernels}
In this section we discuss a construction of $\Ho$-inner functions via a different approach. The ideas here were motivated by the study of contractive zero-divisors on Bergman spaces in \cite[Section~5.4]{DurenAMS2004} and \cite[Section~3.6]{HedenmalmSpringer2000}. Our approach involves the reproducing kernels of certain weighted spaces. The use of such reproducing kernels to construct contractive divisors in Bergman spaces has also been used by Hansbo \cite{Hansbo1996}. Our construction here works for general $\Ho$-spaces.

Let $b$ belong to $\Ho$ and define $\Ho(|b|^2)$ to be the closure of all holomorphic polynomials $f$ with respect to the norm
\[\|f\|_{\Ho(|b|^2)} = \|fb\|_{\Ho}.\] 
The inner product in $\Ho(|b|^2)$ is defined on polynomials as
\[\inner{f}{g}_{\Ho(|b|^2)} = \inner{fb}{gb}_{\Ho}.\]
Throughout this section, to reduce confusion, we shall use $\inner{\cdot}{\cdot}_{\Ho}$ to denote the inner product on $\Ho$.
It can be shown that for any ${\lambda}\in\D$, the evaluation functional $f\mapsto f({\lambda})$ is bounded on $\Ho(|b|^2)$ so $\Ho(|b|^2)$ is a reproducing kernel Hilbert space of holomorphic functions on $\D$. 
In addition, for any $f\in\Ho(|b|^2)$, we compute
\begin{align*}
\|zf\|_{\Ho(|b|^2)} & = \|zfb\|_{\Ho} \leq \|M_{z}\|_{\Ho\rightarrow\Ho}\cdot\|fb\|_{\Ho} = \|M_{z}\|_{\Ho\rightarrow\Ho}\cdot\|f\|_{\Ho(|b|^2)}.
\end{align*}
Consequently, the multiplication operator $M_{z}$ is bounded on $\Ho(|b|^2)$.

For a moment, suppose now $b$ is a finite Blaschke product with simple zeros $\{z_1,\ldots, z_s\}\subset\D\backslash\{0\}$ and a zero of multiplicity $d$ at the origin. Let $B$ be the $\Ho$-analogue of $b$ constructed in Theorem \ref{T:directsum_inner}. We have seen that $B/b$ is holomorphic on a neighborhood of the closed unit disk and that $B$ can be written in the form
\[B = p + \sum_{j=1}^{s}c_jK_{z_j},\] where $p$ is a polynomial of degree $d$. We claim that $B/b$ is a constant multiple of the reproducing kernel of $\Ho(|b|^2)$ at the origin. Indeed, we take any polynomial $f$ and compute
\begin{align*}
\inner{f}{B/b}_{\Ho(|b|^2)} & = \inner{fb}{B}_{\Ho} = \inner{fb}{p+\sum_{j=1}^{s}c_jK_{z_j}}_{\Ho}\\
& = \inner{fb}{p}_{\Ho} + \sum_{j=1}^{s}\bar{c_j}f(z_j)b(z_j) = f(0)\inner{b}{p}_{\Ho}.
\end{align*}
The last equality follows from the fact that $b(z_j)=0$ for all $1\leq j\leq s$ and that the origin is a zero of $b$ of multiplicity $d$ and $p$ is a polynomial of degree $d$. In addition, $\inner{b}{p}_{\Ho}\neq 0$.
Consequently,
\[f(0) = \binner{f}{\frac{1}{\inner{p}{b}_{\Ho}}\cdot\frac{B}{b}}_{\Ho(|b|^2)},\]
which shows that $\frac{1}{\inner{p}{b}_{\Ho}}\cdot\frac{B}{b}$ is the reproducing kernel function at the origin for $\Ho(|b|^2)$. 

In the case $\Ho=A^2$, the above result is well known. It is also known that $B$ does not process any extraneous zeros, that is, $B/b$ does not vanish in the closed unit disk. See \cite[Section~5.4]{DurenAMS2004} for a detailed discussion for general $A^p$ spaces. 

When $b$ has repeated zeros, the above assertion remains valid. One needs to consider derivatives of reproducing kernel functions as well. We summarize the result in the following theorem.

\begin{theorem}
\label{T:Ho_inner_analogue}
Let $b$ be a finite Blaschke product and $B$ be its $\Ho$-analogue constructed in Theorem \ref{T:directsum_inner}. Then $B/b$ is a constant multiple of the reproducing kernel function of $\Ho(|b|^2)$ at the origin.
\end{theorem}

Now consider the case $b$ is an arbitrary $H^2$-inner function. Even in the case $b$ is an infinite Blaschke product, it is not clear what should be considered as an $\Ho$-analogue of $b$. One may consider infinite sums of kernel functions but convergence may pose a serious issue. Inspired by Theorem \ref{T:Ho_inner_analogue}, we provide here an answer, even for more general $b$.

\begin{theorem}
\label{T:inner_kernel}
Suppose $b$ belongs to $\Ho$. Then the following statements hold.
\begin{enumerate}[(a)]
   \item Let $R_0$ denote the reproducing kernel at the origin for $\Ho(|b|^2)$. Then the function \begin{align}
\label{Eqn:inner_kernel}
u=\frac{bR_0}{\|bR_0\|_{\Ho}}
\end{align} is $\Ho$-inner. 
   \item Suppose $\varphi$ is an $\Ho$-inner function and $F$ is holomorphic such that $b=\varphi F$ and $[b]_{\Ho}=[\varphi]_{\Ho}$. Then $F$ does not vanish on $\D$ and $u=\big(|F(0)|/\overline{F(0)}\big)\varphi$.
\end{enumerate}
\end{theorem}

\begin{proof}
(a) For any integer $m\geq 1$, we have
\begin{align*}
\inner{bR_0 z^m}{bR_0}_{\Ho} & = \inner{R_0z^m}{R_0}_{\Ho(|b|^2)} = R_0(0)\cdot(0)^m = 0.
\end{align*}
We have used the fact that $R_0z^m$ is an element of $\Ho(|b|^2)$. Consequently, $u$ is $\Ho$-inner.

(b) Suppose now that $b=\varphi F$ with $[b]_{\Ho}=[\varphi]_{\Ho}$. There then exist sequences of polynomials $\{q_n\}$ and $\{p_n\}$ such that $bq_n\rightarrow \varphi$ and $\varphi p_n\rightarrow b$ in $\Ho$. That is, $\varphi Fq_n\rightarrow\varphi$ and $\varphi p_n\rightarrow\varphi F$ in $\Ho$. It follows that $\varphi Fq_n\rightarrow \varphi$ and $\varphi p_n\rightarrow \varphi F$ uniformly on compact subsets. The standard argument using the Maximum Modulus Principle and the fact that the zeros of $\varphi$ are isolated implies that $Fq_n\rightarrow 1$ and $p_n\rightarrow F$ uniformly on compact subsets as well. In particular, $F$ does not vanish on $\D$ and $q_n\rightarrow 1/F$ on compact sets.

On the other hand, we have
\begin{align*}
\|q_n-q_m\|_{\Ho(|b|^2)} & = \|bq_n-bq_m\|_{\Ho}\rightarrow 0
\end{align*}
as $m,n\rightarrow\infty$ because $\{bq_n\}$ converges in $\Ho$. This implies that $\{q_n\}$ is a Cauchy sequence, hence converges, in $\Ho(|b|^2)$. Since $\{q_n\}$ converges pointwise to $1/F$, we conclude that $1/F$ must be the limit. Consequently, $1/F$ belongs to $\Ho(|b|^2)$.

We now show that $R_0 = 1/(\overline{F(0)}F)$. For any polynomial $q$, we compute
\begin{align*}
\inner{q}{1/F}_{\Ho(|b|^2)} & = \inner{qb}{b/F}_{\Ho} = \inner{qb}{\varphi}_{\Ho}\\
& = \lim_{n\to\infty}\inner{q\varphi p_n}{\varphi}_{\Ho} = \lim_{n\to\infty}q(0)p_n(0) = q(0)F(0).
\end{align*}
Since polynomials are dense in $\Ho(|b|^2)$, it follows that $1/(\overline{F(0)}F)$ is the reproducing kernel at the origin for $\Ho(|b|^2)$. That is, $R_0 = 1/(\overline{F(0)}F)$. Formula \eqref{Eqn:inner_kernel} then gives
\[u=\frac{bR_0}{\|bR_0\|_{A^2}} = \frac{|F(0)|}{\overline{F(0)}}\,\varphi\]
as desired.
\end{proof}

Let us now consider examples in the Hardy and Bergman spaces. We shall see that for these spaces, formula \eqref{Eqn:inner_kernel} recovers the inner part of $b$.

\begin{example} 
Consider $\Ho=H^2$ and $b$ is any function in $H^2$. Let $b=\varphi\cdot F$ be the inner-outer factorization of $b$. Since the hypothesis of Theorem \ref{T:inner_kernel}(b) is satisfied, we see that formula \eqref{Eqn:inner_kernel} gives us a unimodular constant multiple of the inner function $\varphi$.
\end{example}

\begin{example}
Now consider the Bergman space $\Ho=A^2$ and $b\in A^2$. By \cite[Section 9.2, Theorem 4]{DurenAMS2004} or \cite[Section~3.6]{HedenmalmSpringer2000}, we may write $b=\varphi\cdot F$, where $\varphi$ is $A^2$-inner, $F$ is $A^2$-outer (in the sense that $[F]=A^2$) and $[b]_{A^2}=[\varphi]_{A^2}$. Theorem \ref{T:inner_kernel}(b) again shows that formula \eqref{Eqn:inner_kernel} recovers $\varphi$, up to a unimodular constant factor.
\end{example}

\begin{remark}
It has been known that inner-outer factorization is not unique in the Bergman space, see \cite[Chapter~8]{HedenmalmSpringer2000}. On the other hand, Theorem \ref{T:inner_kernel}(b) asserts that a decomposition $b=\varphi\cdot F$ with $\varphi$ an $\Ho$-inner function must be unique (up to a unimodular constant) if we also require $[b]_{\Ho}=[\varphi]_{\Ho}$.
\end{remark}


\def\cprime{$'$}
\providecommand{\bysame}{\leavevmode\hbox to3em{\hrulefill}\thinspace}
\providecommand{\MR}{\relax\ifhmode\unskip\space\fi MR }
\providecommand{\MRhref}[2]{%
  \href{http://www.ams.org/mathscinet-getitem?mr=#1}{#2}
}
\providecommand{\href}[2]{#2}

\end{document}